\documentclass[12pt,leqno]{amsart}
\usepackage{amsmath,amsfonts,amssymb,amscd,amsthm,amsbsy,upref}
\numberwithin{equation}{section}
\def\hangbox to #1 #2{\vskip3pt\hangindent #1\noindent \hbox to #1{#2}$\!\!$}

\newtheorem{thm}{Theorem}[section]

\newtheorem{lem}[thm]{Lemma}
\newtheorem{cor}[thm]{Corollary}
\newtheorem{prop}[thm]{Proposition}
\theoremstyle{definition}

\theoremstyle{remark}

\def\N{{\mathbb N}}
\def\R{{\mathbb R}}

\newcommand{\trivert}{|\!|\!|}

\def\sfrac#1#2{\kern.1em\raise.5ex\hbox{$#1$}
        \kern-.1em/\kern-.05em\lower.25ex\hbox{$#2$}}

\def\vp{\varepsilon}

\def\Id{\operatorname{Id}}

\newcommand{\kle}{\!<\!}

\newcommand\kin{\!\in\!}

\newcommand{\fw}{\text{\fw}}

\def\cF{{\mathcal F}}

\begin{document}
\allowdisplaybreaks
\title{Equilateral sets in uniformly smooth Banach spaces.}
\author{D. Freeman}
\author{E. Odell}
\thanks{Edward Odell (1947-2013). The author passed
away during the production of this paper.} \author{B. Sari}
\author{Th. Schlumprecht}

\address{Department of Mathematics and Computer Science\\
Saint Louis University , St Louis, MO 63103  USA}
\email{dfreema7@slu.edu}

\address{Department of Mathematics\\ The University of Texas at Austin, Austin, TX 78712-0257}
\email{odell@math.utexas.edu}

\address{Department of Mathematics, University of North Texas, Denton,
TX 76203-5017} \email{bunyamin@unt.edu}

\address{Department of Mathematics, Texas A\&M University, College Station,
TX 77843-3368} \email{schlump@math.tamu.edu}

\thanks{Research of the first, second, and fourth author was supported by the
  National Science Foundation.}
  \thanks{ Research of the third author was
  supported by the Simons Foundation.
} 

\begin{abstract}
Let $X$ be an infinite dimensional uniformly smooth Banach space. We
prove that $X$ contains an infinite equilateral set. That is, there
exists a constant $\lambda>0$ and an infinite sequence
$(x_i)_{i=1}^\infty\subset X$ such that $\|x_i-x_j\|=\lambda$ for
all $i\neq j$.
\end{abstract}

\maketitle


\section{Introduction}\label{S:0}

A subset $S$ of a Banach space $X$ is called equilateral if there
exists a constant $\lambda>0$ such that $\|x-y\|=\lambda$ for all
$x,y\in S$ with $x\neq y$.  Much of the research on equilateral sets
in Banach spaces is in estimating the maximal size of equilateral
sets in finite dimensional Banach spaces, for some examples see
\cite{AP},\cite{MV},\cite{P},\cite{S}, and \cite{SV}.
Much less research has been done on equilateral sets in infinite
dimensional Banach spaces. Instead of estimating the maximal size of
equilateral sets in finite dimensional spaces, we consider the
question of whether or not an infinite equilateral set exists in
some given infinite dimensional Banach space.  That is, given an
infinite dimensional Banach space $X$, does there exist a sequence
$(x_n)_{n=1}^\infty\subset X$ and a constant $\lambda>0$ such that
$\|x_n-x_m\|=\lambda$ for all $n\neq m$?  For example, any
subsymmetric basis is equilateral, such as the unit vector basis for
$\ell_p$ for all $1\leq p<\infty$ or the unit vector basis for
Schlumprecht's space. On the other hand, the unit vector bases for
Tsirelson's space and the hereditarily indecomposable Gowers-Maurey
space are not subsymmetric, and yet they each have equilateral
subsequences. Whether or not a given infinite dimensional Banach
space contains an equilateral sequence is an isometric property.
That is, it is possible for two infinite dimensional Banach spaces
to be linearly isomorphic, and yet only one of them contain an
equilateral sequence. Indeed, Terenzi constructed an equivalent norm
$\trivert\cdot\trivert$ on $\ell_1$ such that the Banach space
$(\ell_1,\trivert\cdot\trivert)$ does not contain an equilateral
sequence \cite{T1},\cite{T2}.  Terenzi gave two distinct renormings
of $\ell_1$ which do not contain an equilateral sequence, and these
are the only known infinite dimensional Banach spaces which do not
contain an equilateral sequence.  However, every renorming of $c_0$
does contain an equilateral sequence \cite{MV}.  Taken together,
these two results are somewhat surprising as both $\ell_1$ and $c_0$
are not distortable.  We show that every uniformly smooth infinite
dimensional Banach space contains an equilateral sequence.

\section{Asymptotic stability}\label{S:1}

Given a uniformly smooth Banach space $X$, before we can construct
an equilateral sequence in $X$, we will need to first construct a
sequence which is very close to being equilateral in certain ways.
 In this section we show how certain properties of weakly null
sequences can be stabilized to make them ``almost equilateral''.

Let $X$ be a uniformly smooth Banach space.  For all $x\in
X\setminus \{0\}$, there exists a unique functional $\phi_x\in
S_{X^*}$ such that $\phi_x(x)=\|x\|$. Furthermore, the map
$\Phi:X\setminus\{0\}\rightarrow S_{X^*}$ given by $\Phi(x)= \phi_x$
is uniformly continuous on subsets of $X$ which are bounded away
from 0. 
The following lemma allows us to choose a sequence which is
asymptotically equilateral. Note that if $X$ is a uniformly smooth
Banach space and $(x_i)_{i=1}^\infty\subset S_X$ is a normalized
weakly null sequence then $(x_i)_{i=1}^\infty$ has a subsequence
which generates a spreading model $(e_i)_{i=1}^\infty$ with
$\|e_1-e_2\|>1$.
 We recall that the {\em  spreading  model}  generated by  a semi
normalized sequence $(x_i)_{i=1}^\infty$ in a Banach space $X$ is a
Banach space $(E,\|\cdot\|)$ with a basis $(e_i)_{i=1}^\infty$
satisfying
$$\Big\|\sum_{i=1}^n a_i e_i\Big\|=\lim_{k_1\to\infty }\lim_{k_2\to\infty }\ldots \lim_{k_n\to\infty }\Big\|\sum_{i=1}^n a_i x_{k_i}\Big\|,
\text{ for all $n\in\N$ and scalars $(a_i)_{i=1}^n$.}$$
\begin{lem}\label{L:1}
Let $X$ be an infinite dimensional Banach space, and let
$(x_i)_{i=1}^\infty\subset S_X$ be a normalized weakly null sequence
with a spreading model $(e_i)_{i=1}^\infty$ such that
$\|e_1-e_2\|=\lambda>1$. There exists a subsequence
$(y_i)_{i=1}^\infty$ of $(x_i)_{i=1}^\infty$ and a sequence of
scalars $(a_i)_{i=1}^\infty\subset \R$ such that $a_i\rightarrow 1$,
and $\lim_{i\rightarrow\infty}\|a_k y_k- a_i y_i\|=\lambda$ for all
$k\in\N$.
\end{lem}

\begin{proof}

For all $x\in X$, we let $\phi_x\in S_X$ be a functional such that
$\phi_x(x)=\|x\|$. We have that
$\lim_{n\rightarrow\infty}\lim_{m\rightarrow\infty}\|x_n-x_m\|=\|e_1-e_2\|=\lambda>1$.
Let $\vp>0$ be chosen so that $\lambda>1+\vp$. By passing to a
 subsequence of $(x_i)_{i=1}^\infty$, we may assume for all $n\in\N$ that
$\lambda_n:=\lim_{m\rightarrow\infty}\|x_n-x_m\|>1+\vp$. Moreover,
we may assume that $\|x_n-x_m\|>1+\vp$ for all $n,m\in\N$. If
$\lambda_n=\lambda$ for all $n\in\N$ then setting $a_n=1$ for all
$n\in\N$ gives us our desired sequence. Thus, after passing to a
subsequence again, we may assume that either $\lambda_n<\lambda$ for
all $n\in\N$ or that $\lambda_n>\lambda$ for all $n\in\N$.

We first consider the case that $\lambda_n<\lambda$ for all
$n\in\N$. We have for all $n,m\in\N$ that $\|x_n\|=\|x_m\|=1$,
$\|\phi_{x_n-x_m}\|=1$, and
$\phi_{x_n-x_m}(x_n-x_m)=\|x_n-x_m\|>1+\vp$.  Thus,
$\phi_{x_n-x_m}(x_n)>\vp$ for all $n,m\in\N$. Let
$\overline{a}_n=1+(\lambda-\lambda_n)/\vp$.  Thus
$\overline{a}_n\rightarrow 1$.   By the definition of spreading
model, we have that $\lim_{m\rightarrow\infty}\|a x_n- x_m\|$ exists
for all $n\in\N$ and $0\leq a \leq \overline{a}_n$. We have that,
\begin{align*}
\lim_{m\to\infty}\|\overline{a}_n x_n-x_m\|&\ge \lim_{m\rightarrow\infty}\phi_{x_n-x_m}(\overline{a}_n x_n-x_m)\\
&= (\overline{a}_n-1)\lim_{m\rightarrow\infty}\phi_{x_n-x_m}(x_n)+
\lim_{m\rightarrow\infty}\phi_{x_n-x_m}(x_n-x_m)\\
&\geq
 \lambda-\lambda_n+\lambda_n=\lambda.
\end{align*}
 Thus, for all $n\in\N$, we have that $\lim_{m\to\infty}\|x_n-x_m\|=\lambda_n<\lambda\leq\lim_{m\to\infty}\|\overline{a}_n
 x_n-x_m\|$.
Hence, we may choose by the Intermediate Value Theorem, applied to
the function $a\mapsto \lim_{m\to\infty } \|ax_n-x_m\|$,
 a constant $1< a_n\leq
\overline{a}_n$  to yield $\lim_{m\rightarrow\infty}\|a_n
x_n-x_m\|=\lambda$. As $\overline{a}_n\rightarrow 1$, we have that
$a_n\rightarrow 1$, and hence $\lim_{m\rightarrow\infty}\|a_n
x_n-a_m x_m\|=\lim_{m\rightarrow\infty}\|a_n x_n-x_m\|=\lambda$ for
all $n\in\N$.

We now consider the case that $\lambda_n>\lambda$ for all $n\in\N$.
By the  definition of spreading models $\lim_{m\rightarrow\infty}\|a
x_n- x_m\|$ exists for all $n\in\N$ and $0\leq a \leq 1$. As
$\|x_m\|= \|0\cdot x_n-x_m\|=1$ and $\lim
_{m\rightarrow\infty}\|x_n-x_m\|>\lambda$,   there exist by the
Intermediate Value Theorem $0<a_n<1$ so that $\lim
_{m\rightarrow\infty}\|a_n x_n-x_m\|=\lambda>1+\vp$.  After passing
to a subsequence of $(x_i)_{i=1}^\infty$, we may assume that
$\|a_nx_n-x_m\|>1+\vp$ for all $m,n\in\N$.

Since  for all $m, n\in\N$ we have  $\|x_n\|=\|x_m\|=1$ and
$\|a_nx_n-x_m\|>1+\vp$, it follows that $\phi_{a_nx_n-x_m}(x_n) >
\vp/{a_n}$, and, thus,
\begin{align*}
\lambda=&\lim_{m\rightarrow\infty}\|a_n x_n-x_m\|\\
=& \lim_{m\rightarrow\infty} \phi_{a_n x_n-x_m}(a_n
x_n-x_m)\\
=&\lim_{m\rightarrow\infty}\phi_{a_n
x_n-x_m}(x_n-x_m)-(1-a_n)\phi_{a_n x_n-x_m}(x_n)\\
 \leq& \lim_{m\to\infty} \|x_n-x_m\|- (1-a_n)\vp/{a_n}= \lambda_n -\vp  (1/{a_n}-1).
\end{align*}
Since $\lambda =\lim_{n\to\infty} \lambda_n$ and $0<a_n<1$, for
$n\in\N$,   it follows that $a_n\rightarrow 1$.  Hence,
$\lim_{m\rightarrow\infty}\|a_n x_n-a_m
x_m\|=\lim_{m\rightarrow\infty}\|a_n x_n-x_m\|=\lambda$ for all
$n\in\N$.

\end{proof}

 By perturbing the asymptotically equilateral
sequence given by Lemma \ref{L:1} and passing to a subsequence, we
obtain the following.

\begin{lem}\label{L:2}
Let $X$ be an infinite dimensional uniformly smooth Banach space,
and let $(x_i)_{i=1}^\infty\subset X$ be a semi-normalized weakly
null sequence. There exists a weakly null block sequence
$(z_i)_{i=1}^\infty$ of $(x_i)_{i=1}^\infty$ with
$\lim_{i\rightarrow\infty}\|z_i\|=1$ and a constant $\lambda>1$ such
that $\lim_{i\rightarrow\infty}\|z_k- z_i\|=\lambda$ for all
$k\in\N$ and
$\lim_{k\rightarrow\infty}\lim_{i\rightarrow\infty}\phi_{z_k-z_i}(z_\ell)=0$
for all $\ell\in\N$.
\end{lem}
\begin{proof}
After passing to a subsequence and scaling, we may assume by Lemma
\ref{L:1} that there exists $\lambda>1$ such that
$\lim_{i\rightarrow\infty}\|x_i\|=1$ and
$\lim_{i\rightarrow\infty}\|x_k- x_i\|=\lambda$ for all $k\in\N$.
 By passing to a subsequence using Ramsey's Theorem, we may assume that there exists $(b_\ell)_{\ell=1}^\infty\subset \R$
such that
$\lim_{k\rightarrow\infty}\lim_{i\rightarrow\infty}\phi_{x_k-x_i}(x_\ell)=b_\ell$
for all $\ell\in\N$. Let $x^*$ be a $w^*$ accumulation point of
$\{\phi_{x_k-x_i}\,:\, k,i\in\N\}$.  As $(x_\ell)_{\ell=1}^\infty$
is weakly null, $\lim_{\ell\rightarrow\infty} x^*(x_\ell)=0$.
Hence, $\lim_{\ell\rightarrow\infty} b_\ell=0$.
 If there exists a subsequence
$(j_\ell)_{\ell=1}^\infty$ of $\N$ such that $b_{j_\ell}=0$ for all
$\ell\in\N$ then setting $z_\ell=x_{j_\ell}$ gives our desired
sequence.
 We thus may assume by passing to a subsequence that
$b_{\ell}^2>|b_{\ell+1}|>0$ for all $\ell\in\N$. We set
$v_\ell=x_{2\ell+1}-\frac{b_{2\ell+1}}{b_{2\ell}}x_{2\ell}$. Thus,
$\lim_{k\rightarrow\infty}\lim_{i\rightarrow\infty}\phi_{x_k-x_i}(v_\ell)=0$
for all $\ell\in\N$.  Furthermore,
 $\lim_{\ell\rightarrow\infty}\|v_\ell-x_{2\ell+1}\|=0$ as $b_{\ell}\rightarrow0$ and  $b_{\ell}^2>|b_{\ell+1}|>0$ for all
 $\ell\in\N$.
As $\Phi$ is uniformly continuous on semi-normalized subsets of $X$,
we have that
$\lim_{k\rightarrow\infty}\lim_{i\rightarrow\infty}\phi_{v_k-v_i}(v_\ell)=0$
for all $\ell\in\N$.  After passing to a subsequence of $(v_i)$, we
may assume by Lemma \ref{L:1} that there exists a sequence of
constants $c_\ell\rightarrow 1$ such that
$\lim_{i\rightarrow\infty}\|c_k v_k- c_i v_i\|=\lambda$ for all
$k\in\N$.  As the map $\Phi$ is uniformly continuous on
semi-normalized subsets of $X$ and $c_k\rightarrow1$, we have that
$\lim_{k\rightarrow\infty}\lim_{i\rightarrow\infty}\phi_{c_k v_k-c_i
v_i}(c_\ell v_\ell)=0$ for all $\ell\in\N$. Furthermore, we have
that $(c_k v_k)_{k=1}^\infty$ is weakly null as
$(x_{2k+1})_{k=1}^\infty$ and $(x_{2k})_{k=1}^\infty$ are weakly
null.  Thus letting $z_k=c_k v_k$ for all $k\in\N$ gives our desired
sequence.

\end{proof}

Given a Banach space $X$, recall that the {\em modulus of
smoothness} of $X$ is the function
$\rho_X:[0,\infty)\rightarrow[0,\infty)$ defined by
$$\rho_X(\tau):=\sup\left\{\frac{1}{2}\|x+\tau y\|+\frac{1}{2}\|x-\tau y\|-1\,: x,y\in
S_X\right\}\quad\textrm{ for all }\tau\in[0,\infty).
$$
The modulus of smoothness quantifies the uniform smoothness of
$S_X$, and a Banach space is uniformly smooth if and only if
$\lim_{\tau\rightarrow 0^+}\frac{\rho_X(\tau)}{\tau}=0$.
\begin{lem}\label{L:smoothLimit}
Let $X$ be a uniformly smooth Banach space and let $Y\subseteq X$ be
a subspace.  Let $(x_j)_{j=1}^\infty\subset X$ be a seminormalized
weakly null sequence such that $\lim_{j\rightarrow\infty}\|y-ax_j\|$
exists for all $y\in Y$ and $a\in\R$. Define $\trivert\cdot\trivert$
on $Y\oplus \R$ by
$\trivert(y,a)\trivert=\lim_{j\rightarrow\infty}\|y-ax_j\|$. Then
$Y\oplus\R$ is a uniformly smooth Banach space under the norm
$\trivert\cdot\trivert$ with modulus of smoothness at most the
modulus of smoothness of $X$.
\end{lem}
\begin{proof}
Let $\rho_X:[0,\infty)\rightarrow[0,\infty)$ be the modulus of
smoothness of $X$.  Let $\tau>0$, and $(x,a),(y,b)\in
S_{Y\oplus\R}$. Since $\lim_{j\to\infty}\| x-a x_j\|=1$ and
$\lim_{j\to\infty}\|  y-b x_j\|=1$, we have that,
\begin{align*}
\frac{1}{2}\trivert&(x,a)+\tau(y,b)\trivert+\frac{1}{2}\trivert(x,a)-\tau
(y,b)\trivert-1\\
 &=\lim_{j\rightarrow\infty}\frac{1}{2}||x-a x_j+\tau(y
-b x_j)||+\lim_{i\rightarrow\infty}\frac{1}{2}||x-a x_i-\tau(y -b
x_j)||-1\\
&=\lim_{j\to\infty} \frac{1}{2}\left\Vert\frac{x-a x_j}{\|x-a
x_j\|}+\tau\frac{y -b x_j}{\|y -b
x_j\|}\right\Vert+\frac{1}{2}\left\Vert\frac{x-a x_j}{\|x-a
x_j\|}-\tau\frac{y -b x_j}{\|y -b
x_j\|}\right\Vert-1 \\
 &\leq\rho_X(\tau).
\end{align*}

Thus, $\rho_{Y\oplus\R}(\tau)\leq\rho_X(\tau)$ and hence $Y\oplus\R$
is uniformly smooth under the norm $\trivert\cdot\trivert$.
\end{proof}

\begin{lem}\label{L:extension}
Let $X$ be a uniformly smooth Banach space and let $Y\subseteq X$ be
a subspace.  Let $(x_j)_{j=1}^\infty\subset X$ be a seminormalized
weakly null sequence such that $\lim_{j\rightarrow\infty}\|y-ax_j\|$
exists for all $y\in Y$ and $a\in\R$. Define $\trivert\cdot\trivert$
on $Y\oplus \R$ by
$\trivert(y,a)\trivert=\lim_{j\rightarrow\infty}\|y-ax_j\|$. Then
for all $z,y\in Y$ and $a,b\in\R$,
$$\phi_{(y,a)}((z,b))= \lim_{j\rightarrow\infty}\phi_{y-a x_j}(z-b
x_j).
$$
\end{lem}
\begin{proof}
Let $(y,a)\in S_{Y\oplus\R}$. We have that
$$\phi_{(y,a)}((y,a))=\trivert(y,a)\trivert=\lim_{j\rightarrow\infty}\|y-a
x_j\|=\lim_{j\rightarrow\infty}\phi_{y-a x_j}(y -a x_j).
$$
Let $(z,b)\in S_{Y\oplus\R}$ such that $\phi_{(y,a)}((z,b))=0$.
Assume that $\lim_{j\rightarrow\infty}\phi_{y-a x_j}(z-b x_j)\neq0$.
Thus, there exists $c>0$, $\sigma\in\{-1,1\}$, and a subsequence
$(k_j)_{j\in\N}$ of $\N$ such that $\sigma\phi_{y-a x_{k_j}}(z-b
x_{k_j})\geq c$ for all $j\in\N$. Let $\lambda>0$.
\begin{align*}
\trivert(y,a)+\lambda
\sigma(z,b)\trivert&=\lim_{j\rightarrow\infty}\|y-a
x_j+\lambda\sigma (z -b x_j)\|\\
&\geq \liminf_{j\rightarrow\infty}\phi_{y-a x_{k_j}}(y-a
x_{k_j}+\lambda\sigma
(z- b x_{k_j}))\\
&=\lim_{j\rightarrow\infty}\phi_{y-a x_{k_j}}(y-a
x_{k_j})+\lambda\liminf_{j\rightarrow\infty}\sigma\phi_{y-a
x_{k_j}}(
z- b x_{k_j})\\
&\geq 1+\lambda c.
\end{align*}
Hence, we have that
$$\phi_{(y,a)}(\sigma(z,b))=\lim_{\lambda\rightarrow0}\frac{\trivert(y,a)+\sigma\lambda(z,b)\trivert-\trivert(y,a)\trivert}{\lambda}\geq\frac{(1+\lambda
c)-1}{\lambda}=c.
$$
This is a contradiction as we have assumed that
$\phi_{(y,a)}((z,b))=0$. Thus $\lim_{j\rightarrow\infty}\phi_{y-a
x_j}(z-b x_j)=0$ for all $(z,b)\in \phi_{(y,a)}^{-1}(0)$.  We have
as well that $\lim_{j\rightarrow\infty}\phi_{y-a x_j}(y-a
x_j)=\phi_{(y,a)}((y,a))$.  Thus,
$\lim_{j\rightarrow\infty}\phi_{y-a x_j}(z-b
x_j)=\phi_{(y,a)}((z,b))$ for all $(z,b)\in Y\oplus \R$.
\end{proof}

\section{A uniform version of the Inverse Mapping Theorem}\label{S:2}

Let $d\in\N$ and $U\subset \R^d$ be a compact and convex subset
whose interior contains the origin. We denote by $C^1_0(U,\R^d)$ the
space of all continuously differentiable function $f: U\to \R^d$,
with $f(0)=0$.
   For $f\in C^1_0(U,\R^d)$, let $f_i$ denote  the $i$-th component of $f$, for $i\le d$. The derivative
    function is denoted by $Df$, {\it i.e.},
$$Df: U\to \R^{(d,d)} \quad \xi \mapsto \left[ \frac{\partial f_i}{\partial x_j}(\xi)\right]_{1\le i,j\le d}.$$
$\R^{(d,d)}$ is the space of $d\times d$ matrices.
 Elements of $\R^{(d,d)}$ can be seen as operators on $\ell^d_2$ and  we denote the operator norm
 on $\R^{(d,d)}$ by $\|\cdot\|_2$.  We also denote the Euclidean norm on $\R^d$ by $\|\cdot\|_2$.

 It follows for $f\in C^1_0(U,\R^{(d,d)}) $ that  the map $Df(\cdot)$ lies in $C(U,\R^{(d,d)})$, the space
 of all $\R^{(d,d)}$-valued continuous functions on $U$.  For
 $M\in C(U,\R^{(d,d)})$ we let $\|M\|_\infty=\sup_{\xi} \|M(\xi)\|_2$ and
 for $f\in C^1_0(U,\R^d)$ we let $\|f\|_{(1,\infty)}=\|Df\|_\infty$. Then
 $\|\cdot\|_\infty$ and $\|\cdot\|_{(1,\infty)}$ are norms on $C(U,\R^{(d,d)} )$ and
  $C^1_0(U,\R^d)$ respectively, which turn  $C(U,\R^{(d,d)} )$ and  $C^1_0(U,\R^d)$ into Banach spaces, and the operator
  $$ D: C^1_0(U,\R^d)\to C(U, \R^{(d,d)}), \quad f\mapsto Df,$$
  is an isometric embedding, onto the subspace of continuous  functions
  $$M=[M_{(i,j)}]: U\to \R^{(d,d)}, \quad \xi\mapsto [ M_{(i,j)}(\xi)]_{1\le i,j\le d},$$
  for which the $i$th row, $[ M_{(i,j)}(\cdot)  ]_{1\le j\le d} $ is a conservative vector field, for all $i=1,2\ldots d$.

  From these considerations and the Theorem of Arzela-Ascoli we obtain the
  following compactness criterium.

\begin{prop}\label{P:1} A bounded subset $B\subset C^{1}_0(U,\R^d)$ is relatively norm compact if and only if
 $\{ Df: f\in B\}$ is equicontinuous.
\end{prop}

   For  a decreasing function $\delta(\cdot):(0,1)\to (0,1) $, with $\lim_{\vp\to0}\delta(\vp)=0$,     and a real number  $R>0$
    we let $\cF_{(\delta(\cdot),R)}$ be the set of all $f\in C^1_0(U,\R^d)$ for which
    $ \|Df(0)\|_2\le R$,  $Df(0)$ is invertible, with $\|Df(0)^{-1}\|_2\le R$, for which the modulus of continuity of $Df$ is not larger than $\delta(\cdot)$,
    {\it i.e.} $\| Df(\xi)-Df(\eta)\|_2\le \vp$, for $\xi,\eta\in U$ with $\|\xi-\eta\|_2\leq
    \delta(\vp)$.  Note that $ \cF_{(\delta(\cdot),R)}$ is a closed and bounded set and
     $\{ Df: f\in \cF_{(\delta(\cdot),R)}\}$ is equicontinuous.  Thus, $\cF_{(\delta(\cdot),R)}$ is compact by Proposition
     \ref{P:1}.

We now state and prove a uniform version of the inverse mapping
theorem.  This will be used in proving our main result in Section
\ref{S:3}.

 \begin{cor}\label{C:2} Let $d\in\N$. For all  $R>0$ and decreasing functions
  $\delta(\cdot):(0,1)\to (0,1) $, with $\lim_{\vp\to0}\delta(\vp)=0$, there is an
  $\eta=\eta(\delta(\cdot),R)$, so that for all $f\kin    \cF_{(\delta(\cdot),R)}$ we have
  $\eta B^d \subset f(U)$, where $B^d$ denotes the Euclidean unit ball in $\R^d$.

 \end{cor}

\begin{proof}
Assume our claim was not true. Then we could choose a sequence
$f^{(n)}\subset   \cF_{(\delta(\cdot),R)}$, so that $\frac1n
B^d\not\subset f^{(n)}(U)$, for all $n\in\N$.

As $ \cF_{(\delta(\cdot),R)}$ is compact, we may assume that
$f^{(n)}$ converges in norm to some $f\in  \cF_{(\delta(\cdot),R)}$.
By the Inverse Mapping Theorem $f$ has a continuously differentiable
inverse $f^{-1}$ on some neighborhood  $V\subset U$ of the origin.
Since the sequence $(Df^{(n)})_{n=1}^\infty$ is bounded, the
sequence $(f^{(n)})_{n=1}^\infty$ is equicontinuous and we can find
$\rho>0$ so that that for all $n\in\N$ $f^{(n)}(\rho B^d)\subset V$.
For $n\in\N$ we consider the map
$$ g^{(n)} : \rho B^d\to \R^d,\quad \xi\mapsto f^{-1}\circ f^{(n)}(\xi).$$
The sequence $(g^{(n)})_{n=1}^\infty$ converges in $C_0^1(\rho
B^d,\R^d)$ to the identity. After possibly decreasing
 $\rho$  and passing to a subsequence of the $(g^{(n)})$ we may assume that for all $n\in\N$
 \begin{align}
 \label{E:1.1}
 \|Dg^{(n)}(x)- \Id\|_2 \le\frac12 \text{ and } \|(Dg^{(n)}(x))^{-1} -\Id\|_2\le \frac12 , \text{ for all $x\in\rho B^d$,} \\
 \label{E:1.2}
\big\| g^{(n)}(z)\!-\!(g^{(n)}(x)\!+\!D
g^{(n)}(x)\!\circ\!(z\!-\!x))\big\|_2 \kle \frac18\|z-x\|_2,\text{
for all $x,z\kin \rho B^d$.}
\end{align}
\eqref{E:1.1} can be achieved since  $D{g^{(n)}}(\cdot)$ uniformly
converges to the identity matrix, and \eqref{E:1.2} can be achieved
using the Taylor formula and the equicontinuity of  the sequence
$\big(D{g^{(n)}}(\cdot)\big)$.

 We claim that the image of $\rho B^d$ under $g=g^{(n)},$ $n\in\N$ contains $\frac\rho{4} B^d$.

  Indeed, assume
 $y\in  \frac\rho4B^d$. Choose $x_1=y$ and note that 

\begin{align*}
 \|g^{(n)}(x_1)-y\|_2 &\le
\|g^{(n)}(y)-Dg^{(n)}(0)(y)\|_2+ \|Dg^{(n)}(0)(y)-y\|_2\\
&\le  \frac18 \|y\|_2+\frac12 \|y\|_2\le \frac{\rho}4.
\qquad\textrm{ by \eqref{E:1.2} and \eqref{E:1.1}}.
\end{align*}


Assume that we have chosen $x_1,x_2,\ldots x_m\in \rho B^d$
satisfying the following conditions for all $j=1,2 \ldots m$.
\begin{align}\label{E:1.3}
\|x_j-x_{j-1}\|_2&\le \frac32\Big(\frac14\Big)^{j-1}\rho\text{ \ \   (if $j>1$) and thus }\\
 \|x_j\|_2&\le \frac\rho4+\rho\sum_{i=2}^j  \frac32\Big(\frac14\Big)^{i-1}<\rho, \notag  \\
\label{E:1.4} \|g(x_j) -y\|_2 &\le \Big(\frac1{4}\Big)^{j} \rho .
\end{align}
 Then we let
 $$x_{m+1} = x_m+ \big(Dg(x_m)\big)^{-1} \big(y-g(x_m)\big).$$

 It follows  from \eqref{E:1.1} and the induction hypothesis \eqref{E:1.4} that
 \begin{align}\label{E:1.5}
 \| x_{m+1}-x_m\|_2\le  \big\|(Dg(x_m))^{-1} \big\|_2\cdot \|y-g(x_m)\|_2 \le
 \frac32\Big(\frac14\Big)^{-m}\rho.
 \end{align}
 We now have that
 \begin{align*}
 \|&g(x_{m+1}) -y\|_2  \\
  &\le \underbrace{\big\| g(x_m)+ D g(x_m)\circ(x_{m+1}-x_m)-y\big\|_2}_{=0}\\
  &\qquad+\|g(x_{m+1})-(g({x_m})+Dg(x_m)\circ(x_{m+1}-x_m))\|_2\\
  &\le \frac{1}{8}\|x_{m+1}-x_m\|_2\quad\quad\textrm{ by
  }(\ref{E:1.2})\\
&\le
\frac{1}{8}\frac32\Big(\frac14\Big)^{-m}\rho<\Big(\frac14\Big)^{-(m+1)}\rho
\quad\quad\textrm{ by
  }(\ref{E:1.5}).
 \end{align*}
 which finishes the induction step.

 Letting $x=\lim_{m\to\infty} x_m =x_1+\sum_{j=1}^\infty x_{j+1} -x_j$ it follows that
 $$\|x\|_2\le    \frac{\rho}4  +\frac{3\rho}2\sum_{j=1}^\infty \frac1{4^j}  \le \frac{\rho}8+ \frac32\frac14\frac43\rho<\rho,$$
 and by \eqref{E:1.4} we have $g(x)=y$.  Hence, the image of $\rho B^d$ under $g=g^{(n)},$ $n\in\N$ contains $\frac\rho{4} B^d$.

 Finally we can find a positive $\rho'>0$ so that $\rho'B^d\subset f(\frac\rho4B^d)$, and thus
 $$ \rho' B^d \subset f\Big(\frac\rho4B^d\Big)\subset f \circ g^{(n)}(\rho B^d)= f^{(n)}(\rho B^d)\subset f^{(n)}(U),$$
 which contradicts  $\frac1n
B^d\not\subset f^{(n)}(U)$, for all $n\in\N$, and hence our proof is
complete.
\end{proof}

\section{Constructing an equilateral set}\label{S:3}

Given an infinite dimensional uniformly smooth Banach space $X$, our
goal is to construct an equilateral sequence
$(x_n)_{n=1}^\infty\subset X$.  This will be done by first
constructing a sequence $(z_n)_{n=1}^\infty\subset X$ which is
``close'' to being equilateral as in Section \ref{S:1}. We will then
choose $\vp_n\searrow0$ and perturb $(z_n)_{n=1}^\infty$ by a
triangular array of constants $(a_{i,n})_{1\leq i\leq n<\infty}$
(with $|a_{i,n}|<\vp_n$ for all $1\leq i\leq n$) such that if we set
$x_n= (1+a_{n,n})z_n+\sum_{i=1}^{n-1} a_{i,n} z_i$ then
$(x_n)_{n=1}^\infty$ is equilateral.  The sequence $\vp_n\searrow0$
will be determined by the following lemma.

For $N\in\N$, $(\vp_i)_{i=2}^N\subset [0,1)$, and $1> C>0$ we define
$A_{N\times N}(C,(\vp_i)_{i=2}^N)$ to be the set of $N\times N$
matrices $[a_{i,j}]_{1\leq i,j\leq N}\in\R^{\N\times\N}$ which
satisfy the following three properties.
\begin{enumerate}
\item $|a_{i,j}|\leq 2$ for all $1\leq i,j\leq N$,
\item $|a_{i,i}|\geq C$ for all $1\leq i\leq N$,
\item $|a_{i,j}|\leq\vp_j$  for all $1\leq i<j\leq N$.
\end{enumerate}

\begin{lem}\label{L:UpperT}
For all $C>0$ there exists a sequence
$(R_N)_{N=1}^\infty\subset(0,\infty)$ and
 a sequence $(\vp_i)_{i=2}^\infty\subset(0,1)$ such that $A$ is invertible and
$\|A^{-1}\|\leq R_N$ for all $A\in A_{N\times
N}(C,(\vp_i)_{i=2}^N)$.
\end{lem}
\begin{proof}

We will prove the lemma by induction on $N\in\N$.  For $N=1$ the
lemma holds for $R_1=\frac{1}{C}$. We now let $N\in\N$ and assume
that $(\vp_i)_{i=2}^N$ has been chosen such that if $A\in A_{N\times
N}(C,(\vp_i)_{i=2}^N)$ then $A$ is invertible.  We let
$A'=A_{N+1\times N+1}(C,(\vp_2,\vp_3,...,\vp_N,0)$.

If $[a_{i,j}]_{1\leq i,j\leq N+1}\in A'$ then $[a_{i,j}]_{1\leq
i,j\leq N}\in A_{N\times N}(C,(\vp_i)_{i=2}^N)$ is invertible by the
induction hypothesis, and hence $[a_{i,j}]_{1\leq i,j\leq N+1}$ is
invertible because the last column of $[a_{i,j}]_{1\leq i,j\leq
N+1}$ is linearly independent from the others.  Thus, $A'$ is a
compact set of invertible matrices.  As the set of invertible
matrices on $\R^{N+1}$ is open, there exists $\vp_{N+1}$ such that
$[a_{i,j}+\delta_{i,j}]_{1\leq,i,j\leq N+1}$ is invertible for all
$[a_{i,j}]_{1\leq i,j\leq N+1}\in A'$ with $|\delta_{i,j}|\leq
\vp_{N+1}$ for all $1\leq i,j\leq N+1$.  The map $A\mapsto A^{-1}$
is continuous on the set of invertible matrices, and hence there
exists a constant $R_{N+1}>0$ such that
$\|[a_{i,j}+\delta_{i,j}]_{1\leq,i,j\leq N+1}^{-1}\|\leq R_{N+1}$
for all $[a_{i,j}]_{1\leq i,j\leq N+1}\in A'$ with
$|\delta_{i,j}|\leq \vp_{N+1}$ for all $1\leq i,j\leq N+1$ as this
set is compact.  Thus, $\|A^{-1}\|\leq R_{N+1}$ for all $A\in
A_{(N+1)\times (N+1)}(C,(\vp_i)_{i=2}^{N+1})$.
\end{proof}

We are now ready to prove our main result.

\begin{thm} Let $X$ be an infinite dimensional uniformly smooth
Banach space.  There exists a sequence $(x_i)_{i=1}^\infty\subset X$
and a constant $\lambda>0$ such that $\|x_i-x_j\|=\lambda$ for all
$i\neq j$.
\end{thm}

\begin{proof}

For all $x\in X\setminus \{0\}$, we let $\phi_x\in S_{X^*}$ be the
unique functional such that $\phi_x(x)=\|x\|$.  By Lemma \ref{L:2},
there exists a weakly null sequence $(z_i)_{i=1}^\infty\subset X$
such that $\lim_{i\rightarrow\infty}\|z_i\|=1$ and a constant $2>
\lambda>1$ such that $\lim_{i\rightarrow\infty}\|z_k- z_i\|=\lambda$
for all $k\in\N$ and
$\lim_{k\rightarrow\infty}\lim_{i\rightarrow\infty}\phi_{z_k-z_i}(z_\ell)=0$
for all $\ell\in\N$. We may thus assume that $\|z_k-
z_i\|>(1+\lambda)/2$ for all $i\neq k$ and $\|z_i\|<(3+\lambda)/4$
for all $i\in\N$. This gives us the following estimate for all
$i\neq k$.
\begin{equation}\label{E:lowerBound}
\phi_{z_k-z_i}(z_k)=\phi_{z_k-z_i}(z_k-z_i)+\phi_{z_k-z_i}(z_i)\geq
\|z_k-z_i\|-\|z_i\|>\frac{1+\lambda}{2}-\frac{3+\lambda}{4}=\frac{\lambda-1}{4}.
\end{equation}
We set $C=\frac{\lambda-1}{8}$, and thus we have that
$\phi_{z_k-z_i}(z_k)>2C>0$ for all $i\neq k$.  By Lemma
\ref{L:UpperT} there exists  $(R_N)_{N=1}^\infty\subset(0,\infty)$
and $(\vp_i)_{i=2}^\infty\subset(0,1)$ such that $\|A^{-1}\|\leq
R_N$ for all $A\in A_{N\times N}(C,(\vp_i)_{i=2}^N)$.  By induction
on $N\in\N$, we shall produce a sequence $(x_i)_{i=1}^\infty\subset
X$ and sequences of natural numbers $M_N=(m_i^N)_{i=1}^\infty$ with
$M_0=\N$ and $M_N$ a subsequence of $M_{N-1}$, for all $N\in\N$, so
that for all $N\in\N$, the following properties are satisfied.
\begin{enumerate}
\item $\|x_i-x_j\|=\lambda$ for all $1\leq i<j\leq N$,
\item $\lim_{i\rightarrow\infty}\|x_k- z_{m^N_i}\|=\lambda$
for all $1\leq k\leq N$,
\item $\|x_i\|\leq 2$ for all $1\leq i\leq N$,
\item $\lim_{\ell\rightarrow\infty}\lim_{k\rightarrow\infty}\phi_{z_{m^N_\ell}-z_{m^N_k}}(x_i)=0$
for all $1\leq i\leq N$,
\item $|\phi_{z_{L}-x_{k}}(x_i)|<\vp_k$
for all $1\leq i<k\leq N$ and $L\in M_N$.
\item $|\phi_{z_{L}-x_{k}}(x_k)|>C$
for all $1\leq k\leq N$ and $L\in M_N$.
\end{enumerate}
Note that if we are able to construct such a sequence
$(x_i)_{i=1}^\infty$ by induction, then $(x_i)_{i=1}^\infty$ would
be equilateral by condition $(1)$.  Thus, all we need to do to
complete the proof is to prove the induction argument.  Let $N=1$.
We let $x_1=z_1$ and $M_1=(2,3,4,...)$.  Conditions $(1)$ and $(5)$
are trivially satisfied.  Condition $(2)$, $(3)$, $(4)$, and $(6)$
are satisfied by our choice of $(z_i)_{i=1}^\infty$.

We now let $N\in\N$ and assume that we have constructed
$(x_i)_{i=1}^N$ and $M_N=(m^N_i)_{i=1}^\infty$ to satisfy conditions
$(1)$ through $(6)$.  For each $K\in M_N$, we define a map
$g^K:B_{\R^{N+1}}\rightarrow X$ by
$g^K(a_1,...,a_{N+1})=(1+a_{N+1})z_K+\sum_{i=1}^{N}a_i x_i$.
 Our first goal is to show that there exists $\delta>0$ and a subsequence $M'_N$ of $M_N$ such that
if we set $x_{N+1}=g^K(a)$ for some $a\in \delta B_{\R^{N+1}}$ and
$K\in M'_N$, and if $M_{N+1}$ is some subsequence of $\{L\in M'_N|
L>K\}$, then properties $(3)$, $(4)$, $(5)$, and $(6)$ would all
hold.

As $\|z_K\|\leq (3+\lambda)/4<2$ for all $K\in\N$, we may choose
$\delta_1>0$ such that $\|g^K(a)\|\leq 2$ for all $a\in \delta_1
B_{\R^{N+1}}$ and $K\in M_1$.  Thus, if $x_{N+1}=g^K(a)$ for some
$a\in \delta_1 B_{\R^{N+1}}$ and $K\in M_N$ then $\|x_{N+1}\|\leq 2$
and hence property $(3)$ in the induction hypothesis would be
satisfied.

For each $K\in\N$, we have that
$$\lim_{\ell\rightarrow\infty}\lim_{k\rightarrow\infty}\phi_{z_{m^N_\ell}-z_{m^N_k}}(g^K(a))=\lim_{\ell\rightarrow\infty}\lim_{k\rightarrow\infty}(1+a_{N+1})\phi_{z_{m^N_\ell}-z_{m^N_k}}(z_K)+\sum_{i=1}^{N}a_i
\phi_{z_{m^N_\ell}-z_{m^N_k}}(x_i)=0,
$$
as
$\lim_{k\rightarrow\infty}\lim_{i\rightarrow\infty}\phi_{z_k-z_i}(z_\ell)=0$
for all $\ell\in\N$ and
$\lim_{\ell\rightarrow\infty}\lim_{k\rightarrow\infty}\phi_{z_{m^N_\ell}-z_{m^N_k}}(x_i)=0$
for all $1\leq i\leq N$.  Thus, if $x_{N+1}=g^K(a)$ for some $a\in
\delta_1 B_{\R^{N+1}}$ and $K\in M_N$ then property $(4)$ in the
induction hypothesis would be satisfied.

By $(4)$, there exists a subsequence $M'_N=(m'^N_j)_{j=1}^\infty$ of
$M_N$ such that $|\phi_{z_{L}-z_{K}}(x_i)|\leq \vp_{N+1}/2$ for all
$1\leq i\leq N$ and $L,K\in M'_N$ with $L>K$. The set $(g^K)_{K\in
M'_N}$ is equicontinuous on $\delta_1 B_{\R^{N+1}}$, $g^K(0)=z_K$
for all $K\in M'_N$, and the map $x\mapsto \phi_x$ is uniformly
continuous on $X\setminus \vp B_X$ for all $\vp>0$.  Thus, there
exists $\delta_2>0$ with $\delta_2<\delta_1$ such that
$|\phi_{z_{L}-g^{K}(a)}(x_i)|<\vp_{N+1}$ for all $1\leq i\leq N$ and
$K,L\in M'_N$ with $L>K$ and all $a\in \delta_2 B_{\R^{N+1}}$. Thus,
if $x_{N+1}=g^K(a)$ for some $a\in \delta_2 B_{\R^{N+1}}$ and $K\in
M'_N$ then property $(5)$ in the induction hypothesis would be
satisfied for all $L\in M'_N$ with $L>K$.  Similarly, after passing
to a further subsequence of $M'_N$, we may assume that there exists
$\delta>0$ with $\delta<\delta_2$ such that if $x_{N+1}=g^K(a)$ for
some $a\in \delta B_{\R^{N+1}}$ and $K\in M'_N$ then property $(6)$
in the induction hypothesis would be satisfied for all $L\in M'_N$
with $L>K$.  Thus, if we set $x_{N+1}=g^K(a)$ for some $a\in \delta
B_{\R^{N+1}}$ and $K\in M'_N$, and if $M_{N+1}$ is some subsequence
of $\{L\in M'_N| L>K\}$, then properties $(3)$, $(4)$, $(5)$, and
$(6)$ would all hold.

Our next step is to show that we may choose $a\in \delta
B_{\R^{N+1}}$, $K\in M'_N$, and a subsequence $M_{N+1}$ of $\{L\in
M'_N| L>K\}$ such that properties $(1)$ and $(2)$ hold for
$x_{N+1}=g^K(a)$. For each $K\in M_N$ we define a map
$f:\R^{N+1}\rightarrow \R^{N+1}$ by
$$f^K(a)=\left(\|g^K(a)-x_1\|,...,\|g^K(a)-x_N\|,\lim_{j\rightarrow\infty}\|g^K(a)-z_{m'^N_j}\|\right)
\quad\textrm{for all }a\in\R^{N+1}.$$

The derivative of $f$ at $0$, is given by
$Df^K(0)=\left[\frac{\partial f^K_j}{\partial
a_n}|_{a=0}\right]_{1\leq j,n\leq N+1}$. For $1\leq j,n\leq N$,
\begin{equation}\label{E:D1}
\left.\frac{\partial f^K_j}{\partial
a_n}\right|_{a=0}=\left.\frac{\partial}{\partial a_n}\right|_{a=0}
\left\|(1+a_{N+1})z_K+\sum_{i=1}^{N}a_i
x_i-x_j\right\|=\phi_{z_K-x_j}(x_n).\end{equation}
 For $1\leq n\leq N$, we have by
Lemma \ref{L:extension}
\begin{equation}\label{E:D2}
\left.\frac{\partial f^K_{N+1}}{\partial
a_n}\right|_{a=0}=\left.\frac{\partial}{\partial a_n}\right|_{a=0}
\lim_{j\rightarrow\infty}\left\|(1+a_{N+1})z_K+\sum_{i=1}^{N}a_i
x_i-z_{m'^N_j}\right\|=\lim_{j\rightarrow\infty}
\phi_{z_K-z_{m'^N_j}}(x_n).\end{equation}

For $1\leq j\leq N$, we have
\begin{equation}\label{E:D3}
\left.\frac{\partial f^K_{j}}{\partial
a_{N+1}}\right|_{a=0}=\left.\frac{\partial}{\partial
a_{N+1}}\right|_{a=0} \left\|(1+a_{N+1})z_K+\sum_{i=1}^{N}a_i
x_i-x_j\right\|=\phi_{z_K-x_j}(z_K).\end{equation}

By Lemma \ref{L:extension},
\begin{equation}\label{E:D4}
\left.\frac{\partial f^K_{N+1}}{\partial
a_{N+1}}\right|_{a=0}=\left.\frac{\partial}{\partial
a_{N+1}}\right|_{a=0}
\lim_{j\rightarrow\infty}\left\|(1+a_{N+1})z_K+\sum_{i=1}^{N}a_i
x_i-z_{m'^N_j}\right\|=\lim_{j\rightarrow\infty}
\phi_{z_K-z_{m'^N_j}}(z_K).\end{equation}

We first note that equations \eqref{E:D1}, \eqref{E:D2},
\eqref{E:D3}, and \eqref{E:D4} imply that $\left|\frac{\partial
f^K_{j}}{\partial a_n}|_{a=0}\right|\leq 2$ for all $1\leq j,n\leq
N+1$ and $K\in M_N$ as $\|x_n\|\leq 2$ for all $1\leq n\leq N$ and
$\|z_K\|\leq 2$ for all $K\in M'_N$. By equation \eqref{E:D1} and
property $(5)$, we have that $\left|\frac{\partial f^K_j}{\partial
a_n}|_{a=0}\right|=\left|\phi_{z_K-x_j}(x_n)\right|<\vp_j$ for all
$K\in M_N$ and $1\leq n<j\leq N$. By equation \eqref{E:D2} we have
that $\frac{\partial f^K_{N+1}}{\partial a_n}|_{a=0}=
\lim_{j\rightarrow\infty} \phi_{z_K-z_{m'^N_j}}(x_n)$. Thus, by
property $(4)$, there exists $K_1\in M'_N$ such that
$\left|\frac{\partial f^K_{N+1}}{\partial
a_n}|_{a=0}\right|<\vp_{N+1}$ for all $1\leq n\leq N$ and all $K\in
M'_N$ with $K\geq K_1$. By equation \eqref{E:D1} and property $(6)$,
we have that $\left|\frac{\partial f^K_j}{\partial
a_j}|_{a=0}\right|=|\phi_{z_K-x_j}(x_j)|>C$ for all $K\in M'_N$ and
all $1\leq j\leq N$. By equation \eqref{E:D4}, we have that
$\left|\frac{\partial f^K_{N+1}}{\partial a_{N=1}}|_{a=0}\right|=
\left|\lim_{j\rightarrow\infty} \phi_{z_K-z_j}(z_K)\right|\geq
\frac{\lambda-1}{4}>C$.  Thus, we have that
 $Df^K(0)\in A_{(N+1)\times (N+1)}(C,(\vp_i)_{i=2}^{N+1})$ and hence
$\|(Df^K(0))^{-1}\|\leq R_{N+1}$ for all $K\in M'_N$ with $K\geq
K_1$.

Due to property $(2)$ and $\lim_{i\rightarrow\infty}\|
z_k-z_i\|=\lambda$ for all $k\in\N$, we have that
$\lim_{j\rightarrow \infty}f^{m'^N_j}(0)=(\lambda,...,\lambda)$.  As
$\|(Df^{K}(0))^{-1}\|\leq R_{N+1}$ for all $K\in M'_N$ with $K\geq
K_1$, we may apply Corollary \ref{C:2} to obtain an integer $K\in
M'_N$ with $K\geq K_1$ such that $(\lambda,...,\lambda)\in
f^{K}(\delta B_{N+1})$.  Thus, there exists $a\in \delta
B_{\R^{N+1}}$ such that $f^K(a)=(\lambda,...,\lambda)$.  We set
$x_{N+1}=g^K(a)$ and $M_{N+1}=\{L\in M'_N\,|\, L>K\}$.  As noted
earlier, this choice of $x_{N+1}$ and $M_{N+1}$ satisfies properties
$(3)$, $(4)$, $(5)$, and $(6)$ in the induction hypothesis.
Furthermore, we have that
$$\left\|x_{N+1}-x_j\right\|=\left\|(1+a_{N+1})z_K
+\sum_{i=1}^N a_i x_i-x_j\right\|=\lambda\quad \textrm{ for all
}1\leq i\leq N,$$
 thus satisfying property $(1)$.  We have that
$$\lim_{j\rightarrow\infty}\|x_{N+1}-z_{m^{N+1}_j}\|=\lim_{j\rightarrow\infty}\left\|(1+a_{N+1})z_K
+\sum_{i=1}^N a_i x_i-z_{m^{N+1}_j}\right\|=\lambda,$$ thus
satisfying property $(2)$ in the induction hypothesis.  We have
satisfied all properties in our induction hypothesis, and hence we
obtain a sequence $(x_i)_{i=1}^\infty\subset X$ by induction which
satisfies $\|x_i-x_j\|=\lambda>0$ for all $i\neq j$.

\end{proof}

\end{document}